\documentclass{llncs}

\usepackage{amssymb,amsmath,amsfonts,color,verbatim,amscd,bm} 
\spnewtheorem{assumption}{Assumption}{\bfseries}{\itshape}

\usepackage{pgfplots}
\usepackage{subcaption}
\pgfplotsset{compat=1.6}
\usepackage{cancel}
\usepackage{mathtools}
\providecommand{\norm}[1]{\left\lVert#1\right\rVert}
\newcommand{\normp}[2]{\norm{#1}_{H^{#2}}}
\DeclareMathOperator*{\argmin}{arg\,min}
\newcommand{\Rhn}{R^h_n}
\newcommand{\ha}{\frac{1}{2}}
\newcommand{\phiV}{\boldsymbol{\phi}_{\boldsymbol{V}}}
\newcommand{\phiu}{\phi_u}
\newcommand{\Vh}{\boldsymbol{V}^h}
\newcommand{\V}{\boldsymbol{V}}
\newcommand{\vu}{\boldsymbol{u}}

\newcommand{\vx}{{\boldsymbol{x}}}

\newcommand{\bphih}{\boldsymbol{\phi}^h}
\newcommand{\bpsih}{\boldsymbol{\psi}^h}

\newcommand{\VVh}{\mathcal{V}^h}
\newcommand{\vzero}{\boldsymbol{0}}

\newcommand{\mphi}{\begin{pmatrix} \phiu \\ \phiV \end{pmatrix}}
\newcommand{\mpsi}{\begin{pmatrix} \psi_u \\ {\boldsymbol{\psi}_{\boldsymbol{V}}} \end{pmatrix}}

\newcommand{\dV}{\,\mathrm{d}V}
\newcommand{\dvg}{\nabla\cdot}
\newcommand{\mygrad}{\nabla}

\newcommand{\Vhnn}{\boldsymbol{V}^h_{n+\frac{1}{2}}}
\newcommand{\Vnn}{\boldsymbol{V}_{n+\frac{1}{2}}}
\newcommand{\Vthnn}{\tilde{\boldsymbol{V}}^h_{n+\frac{1}{2}}}

\newcommand{\uhn}{u^h_n}
\newcommand{\uhnn}{u^h_{n+\frac{1}{2}}}
\newcommand{\uthnn}{\tilde{u}^h_{n+\frac{1}{2}}}
\newcommand{\uhnnn}{u^h_{n+1}}
\newcommand{\unn}{u_{n+\frac{1}{2}}}
\newcommand{\vunn}{u_{n+\frac{1}{2}}}
\newcommand{\pnn}{p_{n+\frac{1}{2}}}
\renewcommand{\vec}[1]{\mathbf{#1}}

\begin{document}

\title{Discrete Energy Laws for the First-Order System Least-Squares
  Finite-Element Approach}
\titlerunning{Energy Laws for FOSLS}

\author{J. H. Adler\inst{1} \and I. Lashuk\inst{1} \and
  S. P. MacLachlan\inst{2} \and L. T. Zikatanov\inst{3}}
\institute{Department of Mathematics, Tufts
  University, Medford, MA 02155, USA,\\
  \email{james.adler@tufts.edu, ilya.lashuk@gmail.com}  \and
  Department of Mathematics and Statistics, Memorial University of Newfoundland and Labrador,
  St. John's, Newfoundland and Labrador A1C 5S7, Canada,\\
  \email{smaclachlan@mun.ca} \and
  Department of Mathematics, Penn State, 
  University Park, PA 16802, USA,\\
  \email{ludmil@psu.edu}
}

\maketitle

\begin{abstract}
This paper analyzes the discrete energy laws associated with
first-order system least-squares (FOSLS)
discretizations of time-dependent partial differential equations. Using
the heat equation and the time-dependent
Stokes' equation as examples, we discuss how accurately a
FOSLS finite-element formulation adheres to the underlying energy law
associated with the physical system.  Using regularity arguments involving the initial
condition of the system, we are able to give bounds on the convergence
of the discrete energy law to its expected value (zero in the
examples presented here).  Numerical experiments are
performed, showing that the discrete energy laws hold with order
$\mathcal O\left(h^{2p}\right)$, where $h$ is the mesh spacing and $p$
is the order of the finite-element space.  Thus, the energy law
conformance is held with a higher order than the expected,
$\mathcal{O}\left(h^p\right)$, convergence of the
finite-element approximation.  Finally, we introduce an abstract framework for analyzing the energy laws of general FOSLS discretizations.
\end{abstract}

\section{Introduction}\label{sec:intro}
First-order system least squares (FOSLS) is a finite-element methodology
that aims to reformulate a set of partial differential equations
(PDEs) as a system of first-order
equations \cite{1994CaiZ_LazarovR_ManteuffelT_McCormickS-aa,1997CaiZ_ManteuffelT_McCormickS-aa}.  The problem is posed as a minimization of a functional in
which the first-order differential terms appear quadratically, so that the functional
norm is equivalent to a norm meaningful for the given problem. In
equations of elliptic type, this
is usually a product $H^1$ norm. Some of the compelling features of the FOSLS methodology
include: self-adjoint discrete equations stemming from the
minimization principle; 
good operator conditioning stemming from the use of first-order formulations of the
PDE; and finite-element and multigrid performance that is optimal and
uniform in certain parameters (e.g.,
Reynolds number for the Navier-Stokes equations), stemming from uniform product-norm
equivalence.

Successful FOSLS formulations have been developed for a variety of
applications
\cite{1994BochevP_GunzburgerM-aa,2005BrambleJ_KolevT_PasciakJ-aa}.
One example of a large-scale physical application is in magnetohydrodynamics
(MHD) 
\cite{2011AdlerJ_ManteuffelT_McCormickS_NoltingJ_RugeJ_TangL-aa,2010AdlerJ_ManteuffelT_McCormickS_RugeJ-aa,2010AdlerJ_ManteuffelT_McCormickS_RugeJ_SandersG-aa}.
These numerical methods have led to substantial improvements in MHD
simulation technology; however, several important estimates remain to
be analyzed to confirm their quantitive accuracy.  One of
these is the energy of the system.  Using an energetic-variational
approach
\cite{2005FengJ_LiuC_ShenJ_YueP-aa,1963GelfandI_FominS-aa,1979GiraultV_RaviartP-aa,2010HyonY_KwakD_LiuC-aa},
energy laws of the MHD system can be derived that show that the total
energy should decay as a direct result of the dissipation in the
system.  Initial computations show that the FOSLS method indeed
captures this energy law, but it remains to be shown why it should.


In this paper, we
describe the \emph{discrete} energy laws associated with FOSLS
discretizations of time-dependent PDEs, such as the heat equation or
Stokes' equation, and show quantitatively how they are
related to the continuous physical law.  While we only show results
for these ``simple'' linear systems, the results appear generalizable
to more complicated systems, such as MHD.  Getting the correct
energy law is not only important for numerical stability, but it is
crucial for capturing the correct physics, especially if singularities
or high contrasts in the solution are present.  

The paper is outlined as follows.  In Section \ref{sec:energylaws}, we
discuss the energy laws of a given system and describe their discrete
analogues.  Section \ref{sec:diffusion} analyzes the energy laws
associated with the FOSLS discretizations of the heat equation,
and the same is done for Stokes' equations in Section
\ref{sec:stokes}.  For both examples, we present numerical simulations
in Section \ref{s:numerics}.  Finally, we give some concluding remarks
and generalizations in Section \ref{sec:discussion}.

\section{Energy Laws}\label{sec:energylaws}
The energetic-variational approach (EVA) \cite{2005FengJ_LiuC_ShenJ_YueP-aa,1963GelfandI_FominS-aa,1979GiraultV_RaviartP-aa,2010HyonY_KwakD_LiuC-aa} of hydrodynamic systems in complex fluids is based on the second law of thermodynamics and relies on the fundamental principle that the change in the total energy of a system over time must equal the total dissipation of the system.  This energy principle plays a crucial role in understanding the interactions and coupling between different scales or phases in a complex fluid.  In general, any set of equations that describe the system can be derived from the underlying energy laws.  
The energetic variational principle is based on the energy dissipation law for the whole
coupled system:

\begin{equation}
\label{evalaw} \frac{\partial E_{total}}{\partial t} = -\mathcal{D},
\end{equation}
where $E_{total}$ is the total energy of the system, and $\mathcal{D}$ is the dissipation. 

Simple fluids, where we assume no internal (or elastic) energies, can
also be described in this setting and yield the following energy law:
\begin{equation}
\label{eqn:simplefluid}
\frac{\partial}{\partial t} \left ( \frac{1}{2} \int_{\Omega}
  |\vec{u}|^2\, d\vec{x} \right ) = -\int_{\Omega}
\nu|\nabla\vec{u}|^2\, d\vec{x},
\end{equation}
where $\vec{u}$ represents the fluid velocity and $\nu$ is the fluid
viscosity, accounting for the dissipation in the system.
Applying the so-called least-action principle results in the integral
equation,
\begin{eqnarray}\label{stokesenergyweak}
\left\langle \frac{\partial\vec{u}}{\partial t} + \nabla p, \vec{y} \right\rangle
&=& \left\langle \nabla\cdot\nu\nabla\vec{u},\vec{y} \right\rangle, \forall
\vec{y} \in \mathcal{V},\nonumber
\end{eqnarray}
where we assume an incompressible fluid, $\nabla\cdot u = 0$, and an
appropriate Hilbert space, $\mathcal{V}$.  Here, we use
$\left\langle\cdot,\cdot\right\rangle$ to denote the $L^2(\Omega)$ inner product.
In strong form, we obtain the time-dependent Stokes' equations
(assuming appropriate boundary conditions):
\begin{align}
  \label{eqn:stokes1}
  \frac{\partial \vec{u}}{\partial t} + \nabla p - \nabla\cdot\nu\nabla\vec{u}
&=0,\\
\label{eqn:stokes2}\nabla\cdot\vec{u} &=0.
\end{align}  
Note that the energy law can also be derived directly from the PDE
itself.  First, we consider the weak form of \eqref{eqn:stokes1}-\eqref{eqn:stokes2}, multiplying~\eqref{eqn:stokes1} by $\vec{u}$ and
\eqref{eqn:stokes2} by $p$ and integrate over $\Omega$.
After integration by parts we obtain the following relations:
\begin{eqnarray*}
0&=&  \left\langle \frac{\partial \vec{u}}{\partial t} + \nabla p -
  \nabla\cdot\nu\nabla\vec{u}, \vec{u}\right\rangle + \left\langle
  \nabla\cdot\vec{u},p\right\rangle\\
&=& \left\langle\frac{\partial\vec{u}}{\partial t},\vec{u}\right\rangle +
  \left\langle\nabla p, \vec{u}\right\rangle -
  \left\langle\nabla\cdot\nu\nabla\vec{u},\vec{u}\right\rangle +
  \left\langle\nabla\cdot\vec{u},p\right\rangle\\
&=&  \frac{1}{2}\frac{\partial}{\partial t}\left\langle\vec{u},\vec{u}\right\rangle +
  \left\langle\nabla p,\vec{u}\right\rangle +
  \left\langle\nu\nabla\vec{u},\nabla\vec{u}\right\rangle - \left\langle\vec{u},\nabla
  p\right\rangle.
\end{eqnarray*}
Here, we have assumed that the
boundary conditions are such that the boundary terms, resulting
from the integration by parts, vanish. Hence, we have
\begin{equation*}
  \frac{1}{2}\frac{\partial}{\partial t}\left\langle\vec{u},\vec{u}\right\rangle
  =-\left\langle\nu\nabla\vec{u},\nabla\vec{u}\right\rangle.
\end{equation*}
This approach can also be applied to other PDEs, such as the
heat equation, to show similar energy dissipation relations.  Let $\nu$
be the thermal diffusivity of the body $\Omega$, and $u$ its temperature. Then 
the PDE describing the temperature distribution in $\Omega$
is as follows,
\begin{equation}\label{eqn:heat}
\frac{\partial u}{\partial t} - \nabla\cdot\nu\nabla u = 0,
\quad \mbox{ on } \Omega, \quad u = 0,
\quad \mbox{ on } \partial\Omega.
\end{equation}
As before, we multiply \eqref{eqn:heat} by $u$ and integrate over $\Omega$
to obtain that
\begin{eqnarray*}
0&=&\left\langle \frac{\partial u}{\partial t} -
   \nabla\cdot\nu\nabla u, u\right\rangle =
   \left\langle\frac{\partial u}{\partial t},u\right\rangle -
\left\langle\nabla\cdot\nu\nabla u,u\right\rangle\\
&=&\frac{1}{2}\frac{\partial}{\partial t}\left\langle u,u\right\rangle +
\left\langle\nu\nabla u,\nabla u\right\rangle 
\end{eqnarray*}
Hence, 
\begin{equation*}
\frac{1}{2}\frac{\partial}{\partial
  t}\left\langle u,u\right\rangle =-\left\langle\nu\nabla u,\nabla u\right\rangle,
\end{equation*}
which is the scalar version of \eqref{eqn:simplefluid}.  

For the remainder of the paper, we analyze \eqref{eqn:simplefluid},
specifically how closely the FOSLS method can approximate the energy
law discretely.  We will consider both the scalar (heat equation) and
the vector version (Stokes' equation) in the numerical results, as the
form of the energy law is identical.  
First, we discuss how moving to a finite-dimensional space affects
the energy law. 

\section{Heat Equation}\label{sec:diffusion}
First, we consider the heat equation, assuming a constant diffusion
coefficient $\nu=1$ for simplicity, homogeneous Dirichlet boundary
conditions, and a given initial condition:  
\begin{align}\label{eq:heat}
  \frac{\partial u(\vx,t)}{\partial t} &= \Delta u(\vx,t) \quad \forall\vx\in\Omega, \; \forall t > 0\\
  u(\vx,t) &= 0 \quad \forall \vx\in\partial\Omega, \; \forall t \ge 0 \\
  u(\vx,0) &= u_0(\vx) \quad \forall\vx\in\bar{\Omega}. 
\end{align}
To discretize the problem in time, we consider a symplectic, or
energy-conserving, time-stepping scheme such as Crank-Nicolson.  Given
a time step size, $ \tau $, and time $t_n =  \tau n$, we
approximate $u_n = u(\vec{x},t_n)$ with the following semi-discrete
version of \eqref{eq:heat},
\begin{align*}
\frac{u_{n+1}-u_n}{ \tau } = \frac{\Delta u_{n+1} + \Delta u_n}{2}
\end{align*}

To simplify the calculations later, we introduce an intermediate
approximation, $u_{n+\frac{1}{2}}$, and re-write the semi-discrete
problem as
\begin{gather}
  \begin{aligned}
    \frac{u_{n+\frac{1}{2}}-u_n}{\left( \frac{ \tau }{2} \right)} &= \Delta u_{n+\frac{1}{2}}  \\
    u_{n+\frac{1}{2}}(\vx) &= 0 \quad\forall\vx\in\partial\Omega,\;n=0,1,2,\ldots \\
    u_{n+1} &= 2u_{n+\frac{1}{2}} - u_n
  \end{aligned}
  \label{semid}
\end{gather}
\begin{remark}
To obtain the \emph{semi-discrete} energy law for \eqref{semid}, we
perform a similar procedure as done in Section \ref{sec:energylaws},
where we multiply the first equation in \eqref{semid} by
$u_{n+\frac{1}{2}}$ and integrate over the domain.  After some simple
calculations, we obtain 
the corresponding energy law, using $L^2-$norm notation:
\begin{equation}\label{eq:discrete_energy}
\frac{||u_{n+1}||^2 - ||u_n||^2}{2 \tau } = -||\mygrad u_{n+\frac{1}{2}}||^2
\end{equation}
\end{remark}

To use the FOSLS method, we now put the operator into a first-order
system.  Since we have reduced the problem to a reaction-diffusion
type problem, we introduce a new vector $\vec{V} = \mygrad u$, and use
the $H^1$-elliptic equivalent system \cite{1994CaiZ_LazarovR_ManteuffelT_McCormickS-aa,1997CaiZ_ManteuffelT_McCormickS-aa}:
\begin{equation}\label{FOSLSsystem}
  L_{ \tau } 
  \begin{pmatrix}
    \unn \\
    \Vnn
  \end{pmatrix}
  =
  \begin{pmatrix}
    -\nabla\cdot\Vnn +  \frac{2}{ \tau }\unn \\
    \Vnn-\nabla \unn \\
    \nabla\times\Vnn
  \end{pmatrix}
=
  \begin{pmatrix}
    \frac{2}{ \tau } u_{n} \\
   \vec{0} \\
    \vec{0}
  \end{pmatrix}.
\end{equation}
Note that Dirichlet boundary condition on the continuous solution,
$u$, gives rise to tangential boundary conditions on $\vec{V}$,
$\vec{V}\times\vec{n} = \vec{0}$, where $\vec{n}$ is the normal vector
to the boundary.  

Next, we consider a finite-dimensional subspace of a product $H^1$ space,
$\mathcal{V}^h$, and perform the FOSLS minimization of
\eqref{FOSLSsystem} over $\mathcal{V}^h$:

\begin{align*}
  \left( \uhnn, \Vhnn \right) & = 
  \argmin_{\left( u,\V \right)\in\VVh} 
  \norm{
    L_{ \tau } 
    \begin{pmatrix}
      u \\
      \V
    \end{pmatrix}
    -
    \begin{pmatrix}
      \frac{2}{ \tau } \uhn \\
      \vzero\\
      \vzero\\
    \end{pmatrix}
  }, \\
  \uhnnn & = 2\uhnn-\uhn.
\end{align*}
For each $n$, the above minimization results in the following weak set
of equations:
\begin{equation}\label{FOSLSheat}
  \left\langle
    L_{ \tau } 
    \begin{pmatrix}
      \uhnn \\
      \Vhnn
    \end{pmatrix}
    -
    \begin{pmatrix}
      \frac{2}{ \tau } \uhn \\
      \vzero\\
      \vzero\\
    \end{pmatrix},
    L_{ \tau } \bphih
  \right\rangle = 0 \quad \forall \bphih\in \VVh,
\end{equation}
where the inner products and norms are all in $L^2$ (scalar or vector,
depending on context), unless otherwise noted.

Note, that with the introduction of
$\vec{V}$, the discrete form of the FOSLS energy law can now be written,
\begin{equation}\label{FOSLSenergylaw}
\frac{||u_{n+1}^h||^2 - ||u_n^h||^2}{2 \tau } -||\vec{V}_{n+\frac{1}{2}}^h||^2\to 0,\quad\mbox{as}\quad h\to 0.  
\end{equation}
The goal of the remainder of this Section is to show how well this
energy law is satisfied.
To do so, we make use of the following
assumption.
\begin{assumption}\label{as:initcond}  Assume that the initial condition is smooth enough
  and the projection onto the finite-element space has the following property,
\[  \norm{u_0 - u_0^h}_{H^1}  \le C h^p \norm{u_0}_{H^{p+1}},\]
where $p$ is the order of the finite-element space being considered.
\end{assumption}
Then, using standard regularity estimates we obtain the following Lemma.
\begin{lemma}
  Let $\left\{ u_i \right\}_{i=0,1,\ldots}$ be a sequence of
  semi-discrete solutions to \eqref{semid}. Then, for any successive
  time steps, there exists a constant $C>0$, such that
  \begin{equation*}
    \normp{u_{n+1}}{p} \le C \normp{u_n}{p}
  \end{equation*}
\end{lemma}
A consequence of this regularity estimate is a bound on the error in
the approximation.
\begin{lemma}
  \label{errorest}
  Let $f\in H^{p}\cap H^1_0$ and let the pair $\left( u^h, \Vh \right)\in\VVh$ solve 
  \begin{equation*}
    \begin{pmatrix}
      u^h \\
      \Vh
    \end{pmatrix}
    =
    \argmin_{
      \left(u, \V \right) \in \VVh
    }\norm{
      L_{ \tau } 
      \begin{pmatrix}
	u \\
	\V
      \end{pmatrix}
      -
      \begin{pmatrix}
	\frac{2}{ \tau } f \\
	\vzero\\
	\vzero\\
      \end{pmatrix}
    }^2.
  \end{equation*}
  Let $\hat{u}$ be the exact solution of the corresponding PDE, i.e.,
  \begin{align*}
    -\Delta\hat{u} + \frac{2}{ \tau } \hat{u} &= \frac{2}{ \tau } f \quad\text{in }\partial\Omega, \\
    \hat{u} &= 0 \quad \text{on } \partial\Omega.
  \end{align*}
  Then,
  \begin{equation*}
    \norm{u^h-\hat{u}}_{H^1} \le \frac{C(\tau) h^p}{ \tau } \norm{f}_{H^{p-1}},
  \end{equation*}
  where the constant $C(\tau)$ may also depend on $ \tau $.
\end{lemma}
\begin{proof}
For a fixed $ \tau $, the PDE is a reaction-diffusion equation.
Therefore, standard results from the FOSLS discretization of
reaction-diffusion can be used
\cite{1994CaiZ_LazarovR_ManteuffelT_McCormickS-aa,1997CaiZ_ManteuffelT_McCormickS-aa}.
Note that for a standard FOSLS approach, $C(\tau) =
\mathcal{O}\left(\frac{1}{\tau^2}\right)$, but a rescaling of the
equations may ameliorate this ``worst-case scenario.''
\end{proof}

Next, we make the following observation, which follows from the
well-posedness of the FOSLS formulation \cite{1994CaiZ_LazarovR_ManteuffelT_McCormickS-aa,1997CaiZ_ManteuffelT_McCormickS-aa}.
\begin{lemma}
  \label{rhschange}
  Let $\left( u_1, \V_1 \right)\in\VVh$ and $\left( u_2,\V_2
  \right)\in\VVh$ be two solutions to the following FOSLS weak forms
  with different right-hand sides,
  \begin{equation*}
    \left\langle L_{ \tau } 
      \begin{pmatrix}
        u_1 \\
        \V_1
      \end{pmatrix} - F_1,
      L_{ \tau } \bphih
    \right\rangle = 0,
    \quad
    \left\langle L_{ \tau } 
      \begin{pmatrix}
        u_2 \\
        \V_2
      \end{pmatrix} - F_2,
      L_{ \tau } \bphih
    \right\rangle = 0
    \quad
    \forall \bphih\in\mathcal{V}.
  \end{equation*}
  Then,
  \begin{equation*}
    \norm{u_1-u_2}_{H^1} + \norm{\V_1-\V_2}_{H^1} \le C(\tau) \norm{F_1-F_2}. 
  \end{equation*}
\end{lemma}
This, then, yields the following result.
\begin{lemma}\label{l:accur}
Given the solution to the semi-discrete equation, \eqref{semid}, and
the fully discrete solution, we can bound the error in the $L^2$ norm:
  \begin{equation}
    \norm{\uhnn-u_{n+\frac{1}{2}}} \le \frac{C_1(\tau)}{ \tau } h^p \norm{u_n}_{H^{p}} + C_2(\tau) \norm{\uhn - u_n}.
    \label{accur}
  \end{equation}
\end{lemma}
\begin{proof}
  Let $\uthnn$ be the scalar part of the FOSLS solution
  $\left( \uthnn, \Vthnn \right)$ of
  \begin{equation}
    -\Delta u + \frac{2}{ \tau } u = \frac{2}{ \tau } u_n,
    \text{ in } \Omega,
    \quad
      u = 0 \text{ on } \partial \Omega,
    \label{evobvp}
  \end{equation}
  where the exact semi-discrete solution $u_n$, at the previous time step, is
  used in the right-hand side.  By the triangle inequality,
  \begin{equation}
    \norm{\uhnn-\unn} \le \norm{\uhnn-\uthnn} + \norm{ \uthnn - \unn }.
    \label{accurtri}
  \end{equation}
  By Lemma~\ref{rhschange}, we have
  \begin{equation}
  \renewcommand{\arraystretch}{2}
    \begin{array}{rcl}
    \norm{\uhnn-\uthnn} & \le & 
    \norm{\uhnn-\uthnn}_{H^1}+
    \norm{\Vhnn-\Vthnn}_{H^1}\\
      & \le & C_2(\tau)\norm{\uhn - u_n}.
    \end{array}
    \label{accurfirst}
  \renewcommand{\arraystretch}{1}
  \end{equation}
  The functions $\uthnn$ and $\unn$
  are, respectively, FOSLS and exact solutions of the same boundary
  value problem \eqref{evobvp}. Hence, from Lemma~\ref{errorest}, we have
  \begin{equation}
    \norm{ \uthnn - \unn } \le \frac{C(\tau)}{ \tau } h^{p}
    \norm{u_n}_{H^{p-1}} \leq \frac{C(\tau)}{ \tau } h^{p}
    \norm{u_n}_{H^{p}}
    \label{accursec}
  \end{equation}
Combining \eqref{accurtri}, \eqref{accurfirst} and \eqref{accursec}, we obtain \eqref{accur}.
\end{proof}

Finally, we have the following result on the approximation of the
exact energy law~\eqref{FOSLSenergylaw}.
\begin{theorem}\label{t:elaw}
Let $\left ( \begin{array}{c}u^h_{n}\\\vec{V}^h_n\\\end{array}\right
)$ be the solution to the FOSLS system, \eqref{FOSLSheat}, at time step
$n$ (with $u^h_{n+\frac{1}{2}}$ and $V^h_{n+\frac{1}{2}}$ defined as
before).  There exists $C(\tau)>0$ such that
\[ \left | \frac{\|u^h_{n+1}\|^2 - \|u^h_n\|^2}{2 \tau } +
  \|\vec{V}^h_{n+\frac{1}{2}}\|^2\right | \leq C(\tau) \frac{2}{\tau}\left\|u_n^h-u_n\right\|\min\limits_{ \bphih \in \VVh }
  {
    \left\|
      \begin{pmatrix}
        \uhnn \\
        \Vhnn \\
        \vzero
      \end{pmatrix}
      -L_{\tau} \bphih
    \right\|
  }.
\]
\end{theorem}
\begin{proof}
  To simplify the notation, define the energy law we wish to bound
  as,
  \[E^h_n := \frac{\|u^h_{n+1}\|^2 - \|u^h_n\|^2}{2 \tau } +
  \|\vec{V}^h_{n+\frac{1}{2}}\|^2.\]
Note that
\begin{equation*}
  \ha\frac{\norm{u^h_{n+1}}^2-\norm{u^h_{n}}^2}{ \tau } = \left\langle\frac{\uhnnn-\uhn}{ \tau }, \uhnn \right\rangle = \left\langle
    \frac{\uhnn-\uhn}{\frac{ \tau }{2}}, \uhnn \right\rangle,
\end{equation*}
and
\begin{equation*}
\norm{\Vhnn}^2 = \left\langle-\nabla\cdot\Vhnn,\uhnn\right\rangle + \left\langle\Vhnn-\nabla\uhnn,\Vhnn\right\rangle,
\end{equation*}
where the latter equation is obtained by integration by parts,
continuity of the spaces, and appropriate boundary conditions.
~~Thus,
\begin{multline*}
  E^h_n
  = \left\langle-\nabla\cdot\Vhnn+ \frac{\uhnn-\uhn}{\frac{ \tau }{2}}, \uhnn \right\rangle + 
   \left\langle\Vhnn - \nabla\uhnn, \Vhnn \right\rangle \\
  =
  \left\langle 
    L_{ \tau } 
    \begin{pmatrix}
      \uhnn \\
      \Vhnn
    \end{pmatrix}
    -
    \begin{pmatrix}
      \frac{2}{ \tau } \uhn \\
      \vzero\\
      \vzero\\
    \end{pmatrix},
    \begin{pmatrix}
      \uhnn \\
      \Vhnn \\
      \vzero
    \end{pmatrix}
  \right\rangle.
\end{multline*}
Using \eqref{FOSLSheat}, for any $\bphih \in \mathcal{V}^h$,
\begin{equation*}
  E^h_n = \left\langle 
    L_{ \tau } 
    \begin{pmatrix}
      \uhnn \\
      \Vhnn
    \end{pmatrix}
    -
    \begin{pmatrix}
      \frac{2}{ \tau } \uhn \\
      \vzero\\
      \vzero\\
    \end{pmatrix},
    \begin{pmatrix}
      \uhnn \\
      \Vhnn \\
      \vzero
    \end{pmatrix}
    -
    L_{ \tau } \bphih
  \right\rangle.
  \label{aux1}
\end{equation*}

Next, consider adding and subtracting the solutions to the
semi-discrete, \eqref{FOSLSsystem}, and fully discrete,
\eqref{FOSLSheat}, FOSLS system from the
previous time step,
\begin{multline*}
  E^h_n = \left\langle
    L_{ \tau } 
    \begin{pmatrix}
      \uhnn \\
      \Vhnn
    \end{pmatrix}
    -
    \begin{pmatrix}
      \frac{2}{ \tau } \uhn \\
      \vzero\\
      \vzero\\
    \end{pmatrix} + \frac{2}{\tau}\begin{pmatrix}
      u_n^h - u_n\\
      \vzero\\
      \vzero
    \end{pmatrix},
    \begin{pmatrix}
      \uhnn \\
      \Vhnn \\
      \vzero
    \end{pmatrix}
    -
    L_{\tau} \bphih  
  \right\rangle \\
  - \frac{2}{\tau}\left\langle \begin{pmatrix}
      u_n^h - u_n\\
      \vzero\\
      \vzero
    \end{pmatrix},
    \begin{pmatrix}
      \uhnn \\
      \Vhnn \\
      \vzero
    \end{pmatrix}
    -
    L_{\tau} \bphih  
  \right\rangle\\
  =
  \left\langle
    L_{\tau}\begin{pmatrix}
      \uhnn\\
      \Vhnn\end{pmatrix} - \frac{2}{\tau}\begin{pmatrix}
      u_n\\
      \vzero\\
      \vzero
    \end{pmatrix},
    \begin{pmatrix}
      \uhnn \\
      \Vhnn \\
      \vzero
    \end{pmatrix}
    -
    L_{\tau} \bphih  
  \right\rangle - \frac{2}{\tau}\left\langle \begin{pmatrix}
      u_n^h - u_n\\
      \vzero\\
      \vzero
    \end{pmatrix},
    \begin{pmatrix}
      \uhnn \\
      \Vhnn \\
      \vzero
    \end{pmatrix}
    -
    L_{\tau} \bphih  
  \right\rangle\\
  \leq
  \left\|L_{\tau}\begin{pmatrix}
      \uhnn\\
      \Vhnn\end{pmatrix} - \frac{2}{\tau}\begin{pmatrix}
      u_n\\
      \vzero\\
      \vzero
    \end{pmatrix}\right\| M_n^h + \frac{2}{\tau}M_n^h\left\|u_n^h-u_n\right\|,
\end{multline*}
where we have defined $M_n^h := \min\limits_{ \bphih \in \VVh }
  {
    \left\|
      \begin{pmatrix}
        \uhnn \\
        \Vhnn \\
        \vzero
      \end{pmatrix}
      -L_{\tau} \bphih
    \right\|
  }$.  Then, adding and subtracting $L_{\tau}\begin{pmatrix}
      \unn\\
      \Vnn\end{pmatrix}$ yields
\begin{multline*}
 E_n^h \leq \left\|L_{\tau}\begin{pmatrix}
      \uhnn-\unn\\
      \Vhnn-\Vnn\end{pmatrix} + \cancelto{0}{L_{\tau}\begin{pmatrix}
      \unn\\
      \Vnn\end{pmatrix} - \frac{2}{\tau}\begin{pmatrix}
      u_n\\
      \vzero\\
      \vzero
    \end{pmatrix}}\right\| M_n^h + \frac{2}{\tau}M_n^h\left\|u_n^h-u_n\right\|.
\end{multline*}
Using the continuity of $L_{\tau}$, followed by Lemma
\ref{rhschange}, gives
\begin{multline*}
  |E_n^h| 
  \leq
  C(\tau) M_n^h\left\|\begin{pmatrix}
      \uhnn-\unn\\
      \Vhnn-\Vnn\end{pmatrix} \right\|_{H^1} +
  \frac{2}{\tau}M_n^h\left\|u_n^h-u_n\right\|\\
  \leq
  C(\tau)\frac{2}{\tau}M_n^h\left\|u_n^h-u_n\right\|+
  \frac{2}{\tau}M_n^h\left\|u_n^h-u_n\right\|.
\end{multline*}
Combining the two terms completes the proof.


\end{proof}
To provide a better bound for the FOSLS energy law~\eqref{FOSLSenergylaw}, we
introduce a measure for the truncation error defined as
\begin{equation}\label{e:truncation}
  \delta_n = \max_{v\in H^{p+1}(\Omega)}
  \frac{1}{\|v\|_{H^{p+1}}}\min_{ \bphih \in \VVh }
  {
    \left\|
      \begin{pmatrix}
        \uhnn(v) \\
        \Vhnn(v) \\
        \vzero
      \end{pmatrix}
      -\mathcal{L} \bphih
    \right\|
  },
\end{equation}
where $\uhnn(v)$ and $\Vhnn(v)$ are the corresponding solutions to the
fully discrete problem with $u_0 = v$ as the initial condition.
\begin{corollary}
\label{cor:badbound}
  Using the same assumptions as Theorem \ref{t:elaw} and Assumption
  \ref{as:initcond},
  \begin{equation}
    \left |\frac{\|u^h_{n+1}\|^2 - \|u^h_n\|^2}{2 \tau } +
      \|\vec{V}^h_{n+\frac{1}{2}}\|^2\right | \leq \frac{C(\tau)\delta}{\tau} h^p\|u_0\|_{H^{p+1}}^2, \quad \delta = \max_n\delta_n. 
    \end{equation}
  \end{corollary}
  \begin{proof}
Using the definitions of $\uhnn$ and $\unn$, the triangle inequality,
and Lemma~\ref{l:accur},
\begin{multline*}
  \| u^h_{n+1} - u_{n+1}\| \leq 2\|\uhnn - \unn\| + \|u_n^h - u_n\|\\
  \leq \frac{2}{\tau}C_1(\tau) h^p\|u_{n}\|_{H^p} + \left (C_2(\tau) + 1\right
  )\|u_n^h - u_n\|.
\end{multline*}
An induction argument then gives
\[ \| u_n^h - u_n \| \leq \frac{2}{\tau}C_1(\tau) h^p\sum_{j=1}^n \left
    (C_2(\tau)+1\right )^{j-1}\|u_{n-j}\|_{H^p} + \left (C_2(\tau)+1\right
  )^n\|u_0^h - u_0\|.\]
With Assumption \ref{as:initcond},
\[ \| u_n^h - u_n \| \leq \frac{2}{\tau}C_1(\tau) h^p\sum_{j=1}^n \left
    (C_2(\tau)+1\right )^{j-1}\|u_{n-j}\|_{H^p} + \left (C_2(\tau)+1\right
  )^n\|u_0\|_{H^{p+1}}.\]
Using some regularity arguments for each $u_i$, we get,
\[\|u_n^h - u_n\| \leq C(n) h^p\|u_0\|_{H^{p+1}}.\]
Then, with the definition of $\delta$ and the result from Theorem
\ref{t:elaw}, the proof is complete.
  \end{proof}

We note that the bound in Corollary~\ref{cor:badbound} is a rather
pessimistic one.  At a fixed time, $t$, we expect the quality of both
the fully discrete and semi-discrete approximations to the true
solution to improve as $\tau \rightarrow 0$ and more time-steps are
used to reach time $t$; thus, $\|u_n^h-u_n\|$ should decrease as
$\tau\rightarrow 0$ for $n = t/\tau$.  Furthermore, for the unforced
heat equation, we expect both $u_n^h$ and $u_n$ to decrease in
magnitude with $n$, but this is not accounted for in the bound in
Corollary~\ref{cor:badbound}.  The bound above worsens
with smaller $\tau$ and bigger $n$, showing the limitations of bounding
$\|u_n^h-u_n\|$ by terms depending only on $u_0$ and the finite-element
space.

\begin{remark}
  As shown in the numerical experiments, Section \ref{s:numerics},
  the constant $\delta$ defined in~\eqref{e:truncation} is of order
  $h^p$ for a smooth solution.  This indicates that the energy
  law~\eqref{FOSLSenergylaw} holds with order
  $\mathcal O\left(h^{2p}\right)$. While the theoretical justification
  of such statement may be plausible, it is non trivial as the
  discrete quantities involved in the definition of $\delta$ do not
  possess enough regularity (they are just finite-element functions,
  only in $H^1$).
  \end{remark}

  


\section{Stokes' Equations}\label{sec:stokes}
Next, we return to the time-dependent
Stokes' equations, \eqref{eqn:stokes1}-\eqref{eqn:stokes2}.  For simplicity, we again
assume $\nu=1$, and rewrite the equations using Dirichlet boundary
conditions for the normal components of the velocity field, and zero-mean average for the
pressure field,
\begin{align}\label{eq:stokessys}
  \frac{\partial \vu(\vx,t)}{\partial t} - \Delta \vu(\vx,t) + \nabla p(\vx,t)&= \vec{0}\quad
  \forall\vx\in\Omega, \; \forall t > 0\\
\dvg\vu(\vx,t) & = 0 \quad
  \forall\vx\in\Omega, \; \forall t > 0\\
 \vec{n}\cdot\vu(\vx,t) &= 0 \quad \forall \vx\in\partial\Omega, \; \forall t \ge 0 \\
  \vu(\vx,0) &= \vec{g}(\vx) \quad \forall\vx\in\bar{\Omega}, \\
\int_{\Omega} p(\vx,t) \dV &= 0\quad \forall t \ge 0.
\end{align}

Using a similar semi-discretization in time with Crank-Nicolson that
was done in \eqref{semid} yields,
\begin{gather}
  \begin{aligned}
    \frac{\vu_{n+\frac{1}{2}}-\vu_n}{\left( \frac{ \tau }{2} \right)}
    - \Delta \vu_{n+\frac{1}{2}} + \nabla p_{n+\frac{1}{2}}&= \vec{0},
    \\
\dvg\vu_{n+\frac{1}{2}} &=0,\\
    \vec{n}\cdot\vu_{n+\frac{1}{2}}(\vx) &= 0
    \quad\forall\vx\in\partial\Omega,
    \\
\int_{\Omega} p_{n+\frac{1}{2}} \dV &= 0\quad \forall n \ge 0,\\
    \vu_{n+1} &= 2\vu_{n+\frac{1}{2}} - \vu_n,\\
p_{n+\frac{1}{2}} &=2p_{n+\frac{1}{2}} - p_n.\\
  \end{aligned}
  \label{stokes_semid}
\end{gather}

To use the FOSLS method, we put the operator into a first-order
system in a similar fashion to the heat equation.  Least-squares
formulations are well-studied for Stokes' system and we consider a
simple, velocity-gradient-pressure formulation, where a new gradient tensor,
$\vec{V} = \mygrad \vu$, is used to obtain an
$H^1$-elliptic equivalent system \cite{1998BochevP_CaiZ_ManteuffelT_McCormickS-aa,1999BochevP_ManteuffelT_McCormickS-aa,2007HeysJ_LeeE_ManteuffelT_McCormickS-aa}:
\begin{equation}\label{FOSLSstokes}
  L_{ \tau } 
  \begin{pmatrix}
    \vunn \\
    \Vnn\\
     \pnn
  \end{pmatrix}
  =
  \begin{pmatrix}
    -\nabla\cdot\Vnn + \nabla\pnn +  \frac{2}{ \tau }\unn \\
    \nabla\cdot\vunn \\
    \Vnn-\nabla \vunn \\
    \nabla\times\Vnn \\
    \nabla\textbf{tr}\Vnn
  \end{pmatrix}
=
  \begin{pmatrix}
    \frac{2}{ \tau } \vu_{n} \\
   0 \\
    \vec{0}\\
\vec{0}\\
\vec{0}
  \end{pmatrix}.
\end{equation}

Appropriate boundary condition on the continuous solution, such as
$\vec{n}\cdot\vu=0$, gives rise to tangential boundary conditions on $\vec{V}$,
$\vec{V}\times\vec{n} = \vec{0}$, where $\vec{n}$ is the normal vector
to the boundary. Ultimately, the corresponding semi-discrete energy
law is
\begin{equation}\label{FOSLSenergylawstokes}
\frac{||\vu_{n+1}||^2 - ||\vu_n||^2}{2 \tau } = -||\vec{V}_{n+\frac{1}{2}}||^2.
\end{equation}

Finally, we minimize the residual of \eqref{FOSLSstokes} over a
finite-dimensional subspace of the product $H^1$ Sobolev space in the
$L^2$ norm obtaining the weak equations,
\begin{equation}\label{FOSLSstokesweak}
  \left\langle
    L_{ \tau } 
    \begin{pmatrix}
      \vu^h_{n+\frac{1}{2}} \\
      \Vhnn\\
      p^h_{n+\frac{1}{2}}
    \end{pmatrix}
    -
    \begin{pmatrix}
      \frac{2}{ \tau } \vu^h_n \\
      0 \\
    \vec{0}\\
\vec{0}\\
\vec{0}
    \end{pmatrix},
    L_{ \tau } \bphih
  \right\rangle = 0 \quad \forall \bphih\in \VVh.
\end{equation}
Note, that the weak system is similar to \eqref{FOSLSheat} and the
energy law is identical to \eqref{FOSLSenergylaw} in vector form.
Thus, all the above theory still holds subject to enough regularity of
the solution to the time-dependent Stokes' equations
\cite{VASolonnikov_1964,VASolonnikov_1965} and a suitable
generalization of the definition of $\delta$.

\section{Numerical Experiments}\label{s:numerics}
For the numerical results presented here, we use a C++ implementation
of the FOSLS algorithm, using the modular finite-element library MFEM
\cite{mfem-aa} for managing the discretization, mesh, and
timestepping.  The linear systems are solved by direct method using
the UMFPACK package \cite{2004DavisT-aa}.

\subsection{Heat Equation}
First, we consider the heat equation, \eqref{eq:heat}, and its
discrete FOSLS formulation, \eqref{FOSLSheat}, on a triangulation of
$\Omega = (0,1)\times(0,1)$.  The data is chosen so that the true
solution is $u(x,y,t) = \sin(\pi x)\sin(\pi y)e^{-2\pi^2t}$.  Note that this
solution satisfies the boundary conditions and other assumptions
discussed above.

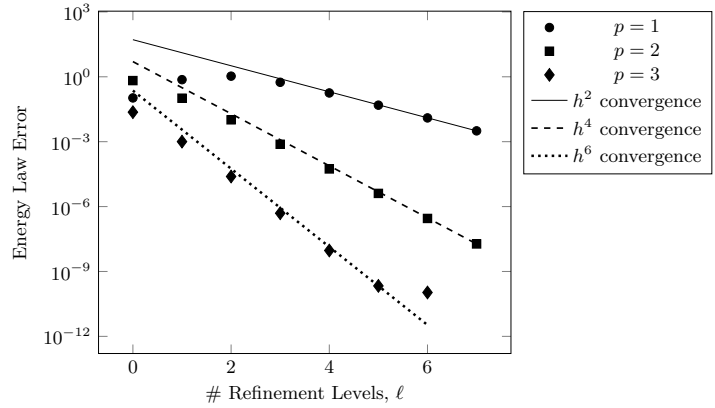
\begin{figure}[h!]
  \centering
  \scalebox{0.8}{
    \begin{tikzpicture}
    \begin{semilogyaxis}[
	xlabel={ \# Refinement Levels, $\ell$ },
	ylabel={Energy Law Error},
        legend pos=outer north east,
        legend style={cells={align=left}}]

        \addplot+[only marks,black,mark=*,mark options={fill=black}] coordinates{
		(0,1.054694e-01)
		(1,7.356504e-01)
		(2,1.070110e+00)
		(3,5.565853e-01)
		(4,1.787960e-01)
		(5,4.876466e-02)
		(6,1.254963e-02)
		(7,3.164519e-03)
      };
      \addlegendentry{$p=1$}
      
        \addplot+[only marks,black,mark=square*,mark options={fill=black}] coordinates{
		(0, 6.667144e-01)
		(1, 1.022337e-01)
		(2, 1.025251e-02)
		(3, 7.599296e-04)
		(4, 5.495789e-05)
		(5, 4.065416e-06)
		(6, 2.824064e-07)
		(7, 1.872939e-08)
      };
      \addlegendentry{$p=2$}

      \addplot+[only marks,black,mark=diamond*,mark options={scale=1.5,fill=black}] coordinates{
		(0,2.314188e-02)
		(1,1.004929e-03)
		(2,2.449755e-05)
		(3,4.922769e-07)
		(4,9.324679e-09)
		(5,2.148797e-10)
		(6,1.066383e-10)
      };
      \addlegendentry{$p=3$}

      \addplot+[ domain=0:7, no markers, solid, black ]
         { 2^(-2*(x-7))*3.164519e-03 };
      \addlegendentry {$h^2$ convergence}

      \addplot+[ domain=0:7, no markers, dashed, black, thick ]
         { 2^(-4*(x-7))*1.872939e-08 };
      \addlegendentry {$h^4$ convergence}

      \addplot+[ domain=0:6, no markers, dotted, black, very thick ]
         { 2^(-6*(x-5))*2.148797e-10 };
      \addlegendentry {$h^6$ convergence}

    \end{semilogyaxis}
  \end{tikzpicture}
  }
  \caption{Energy law error, \eqref{FOSLSenergylaw},  vs. number of
    mesh refinements, $\ell$ ($h=\frac{1}{2^{\ell}}$), for the
    FOSLS discretization of the
    heat equation, \eqref{FOSLSheat}, using various orders of the
    finite-element space ($p=1$ - linear; $p=2$ - quadratic; and $p=3$ -
    cubic). One time step is performed with
    $\tau = 0.005$.}
  \label{fig:heatvsh}
\end{figure}

Figure \ref{fig:heatvsh} displays the convergence of the
energy law to zero as the mesh is refined for a fixed time step.  The
convergence is $\mathcal O\left(h^{2p}\right)$, where $p$ is the order of the finite-element space being
considered, confirming Theorem 1.  It also suggests that the constant
$\delta$ is $\mathcal O\left(h^{p}\right)$, as is remarked above.

\begin{figure}[h!]
  \begin{subfigure}[b]{0.48\textwidth}
    \centering
  \resizebox{\linewidth}{!}{
    \begin{tikzpicture}
    \begin{semilogyaxis}[
	xlabel={ \# Time Steps, $n$ },
	ylabel={Energy Law Error},
        legend pos=north east,
        legend style={cells={align=left}}]

        \addplot+[only marks,black,mark=*,mark options={fill=black}] coordinates{
			(1,4.876466e-02)
			(2,4.039555e-02)
			(4,2.757456e-02)
			(8,1.271522e-02)
			(16,2.673262e-03)
			(32,1.174677e-04)
			(64,2.265981e-07)
      };

       \addplot+[only marks,black,mark=square*,mark options={fill=black}] coordinates{
			(1,4.065416e-06)
			(2,4.112083e-06)
			(4,2.670556e-06)
			(8,1.205922e-06)
			(16,2.482003e-07)
			(32,1.052122e-08)
			(64,1.890633e-11)
      };

       \addplot+[only marks,black,mark=diamond*,mark options={scale=1.5,fill=black}] coordinates{
			(1,2.148797e-10)
			(2,1.935097e-10)
			(4,1.278200e-10)
			(8,5.725109e-11)
			(16,1.173353e-11)
			(32,4.984988e-13)
			(64,8.627607e-16)

      };
    \end{semilogyaxis}
  \end{tikzpicture}
  }
  \caption{}
  \label{fig:heatvsn}
\end{subfigure}
\begin{subfigure}[b]{0.48\textwidth}
    \centering
    \resizebox{\linewidth}{!}{
    \begin{tikzpicture}
    \begin{semilogyaxis}[
	xlabel={ Time Step Size, $\tau$ },
	ylabel={Energy Law Error},
        legend pos=north east,
        legend style={cells={align=left}}]

        \addplot+[only marks,black,mark=*,mark options={fill=black}] coordinates{
		(0.001, 5.422152e-02)
		(0.005, 4.876466e-02)
		(0.01,  4.282118e-02)
		(0.05,  1.740697e-02)
		(0.1,   7.445281e-03)
		(0.5,   2.837192e-04)
		(1,     4.677020e-05)
      };
      \addlegendentry{$p=1$}

       \addplot+[only marks,black,mark=square*,mark options={fill=black}] coordinates{
		(0.001, 3.360948e-06)
		(0.005, 4.065416e-06)
		(0.01,  3.930714e-06)
		(0.05,  1.644762e-06)
		(0.1,   7.004497e-07)
		(0.5,   2.655533e-08)
		(1,     4.374675e-09)
      };
      \addlegendentry{$p=2$}

       \addplot+[only marks,black,mark=diamond*,mark options={scale=1.5,fill=black}] coordinates{
		(0.001, 1.912861e-10)
		(0.005, 2.148797e-10)
		(0.01,  1.944818e-10)
		(0.05,  7.838352e-11)
		(0.1,   3.341394e-11)
		(0.5,   1.222245e-12)
		(1,     1.716266e-13)
      };
      \addlegendentry{$p=3$}

    \end{semilogyaxis}
  \end{tikzpicture}
  }
  \caption{}
  \label{fig:heatvstau}
  \end{subfigure}
\caption{Energy law error, \eqref{FOSLSenergylaw}, vs. (a) number of time steps, $n$ (with fixed
    $\tau=0.005$), and (b) time step
    size, $\tau$, for the
    FOSLS discretization of the
    heat equation, \eqref{FOSLSheat}, using various orders of the
    finite-element space ($p=1$ - linear; $p=2$ - quadratic; and $p=3$ -
    cubic). Mesh spacing is $h=\frac{1}{32}$.}
\label{fig:heatvstime}
\end{figure}
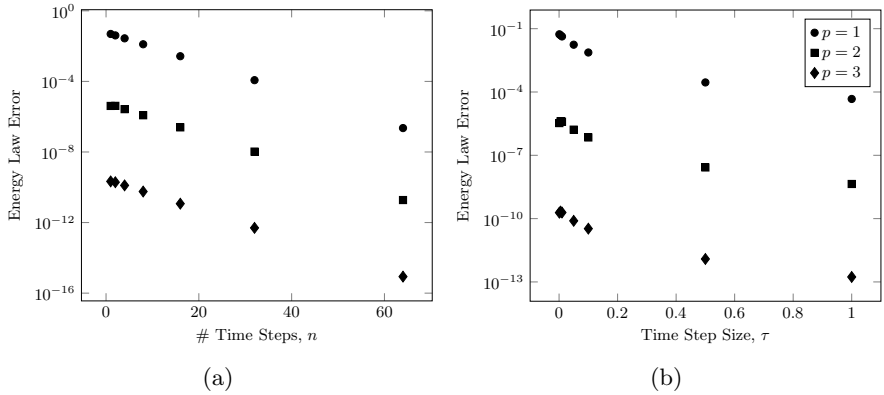

Figure \ref{fig:heatvstime} indicates how the
timestepping affects the convergence of the energy law.  As discussed
above, taking more time steps decreases the error in the energy law,
showing that we can improve the results on the bound, $\|u_n -
u_n^h\|$.  On the other hand, if only one time step is taken, the
convergence slightly worsens for small $\tau$, which is
consistent with the constants found in Theorem \ref{t:elaw} and
Corollary \ref{cor:badbound}.

\subsection{Stokes' Equations}
Next, we consider Stokes' Equations, \eqref{eq:stokessys}, and the
FOSLS discretization described above, \eqref{FOSLSstokesweak}.  The
same domain, $\Omega = (0,1)\times(0,1)$, is used, and we assume data
that yields the exact solution,
\begin{align*}
  \vec{u}(\vec{x},t) = \left ( \begin{array}{c}\sin(\pi x)\cos(\pi
                                 y)\\-\cos(\pi x)\sin(\pi y)\\\end{array}\right )e^{-2\pi^2t},\\
  p(\vec{x},t) = 0.
\end{align*}
This produces a $C^{\infty}$ solution that satisfies the appropriate
boundary conditions and regularity arguments needed for the bounds on
the energy law described above.

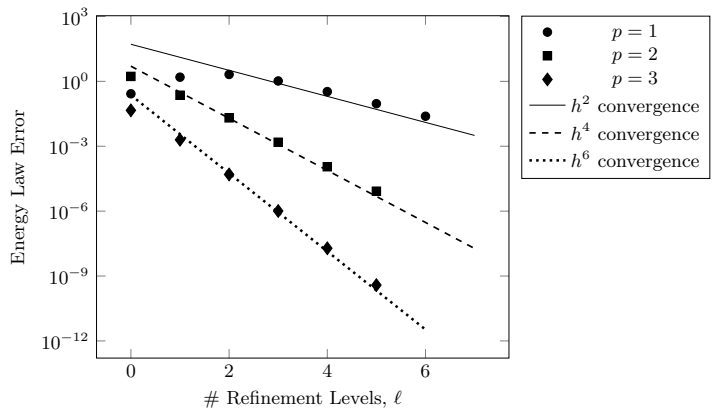
\begin{figure}[h!]
  \centering
  \scalebox{0.8}{
    \begin{tikzpicture}
    \begin{semilogyaxis}[
	xlabel={ \# Refinement Levels, $\ell$ },
	ylabel={Energy Law Error},
        legend pos=outer north east,
        legend style={cells={align=left}}]

        \addplot+[only marks,black,mark=*,mark options={fill=black}] coordinates{
		(0,2.661763e-01)
		(1,1.560125e+00)
		(2,2.084124e+00)
		(3,1.040653e+00)
		(4,3.351039e-01)
		(5,9.281074e-02)
		(6,2.439087e-02)
      };
      \addlegendentry{$p=1$}
      
        \addplot+[only marks,black,mark=square*,mark options={fill=black}] coordinates{
		(0,1.693060e+00)
                (1,2.278559e-01)
		(2,2.053783e-02)
		(3,1.529831e-03)
		(4,1.130386e-04)
		(5,8.333160e-06)
      };
      \addlegendentry{$p=2$}

      \addplot+[only marks,black,mark=diamond*,mark options={scale=1.5,fill=black}] coordinates{
                (0,4.568547e-02)
                (1,1.986097e-03)
		(2,4.972858e-05)
		(3,1.015216e-06)
		(4,1.922519e-08)
		(5,3.816432e-10)
      };
      \addlegendentry{$p=3$}

      \addplot+[ domain=0:7, no markers, solid, black ]
         { 2^(-2*(x-7))*3.164519e-03 };
      \addlegendentry {$h^2$ convergence}

      \addplot+[ domain=0:7, no markers, dashed, black, thick ]
         { 2^(-4*(x-7))*1.872939e-08 };
      \addlegendentry {$h^4$ convergence}

      \addplot+[ domain=0:6, no markers, dotted, black, very thick ]
         { 2^(-6*(x-5))*2.148797e-10 };
      \addlegendentry {$h^6$ convergence}

    \end{semilogyaxis}
  \end{tikzpicture}
  }
  \caption{Energy law error, \eqref{FOSLSenergylawstokes},  vs. number of
    mesh refinements, $\ell$ ($h=\frac{1}{2^{\ell}}$), for the
    FOSLS discretization of the
    Stokes' equation, \eqref{FOSLSstokes}, using various orders of the
    finite-element space ($p=1$ - linear; $p=2$ - quadratic; and $p=3$ -
    cubic). One time step is performed with
    $\tau = 0.005$.}
  \label{fig:stokesvsh}
\end{figure}

Similarly to the heat equation, Figure
  \ref{fig:stokesvsh} compares the convergence of the
energy law to zero as the mesh is refined for a fixed time step.
 Again, we see that the
convergence is $\mathcal O\left(h^{2p}\right)$, where $p$ is the order of the finite-element space being
considered, confirming that Theorem 1 can also be applied to the
time-dependent Stokes' equations.  Thus, the FOSLS
discretization can adhere to the energy
law for fluid-type systems, and has the potential for capturing the
relevant physics of other complex fluids.

\begin{figure}[h!]
  \begin{subfigure}[b]{0.48\textwidth}
    \centering
  \resizebox{\linewidth}{!}{
    \begin{tikzpicture}
    \begin{semilogyaxis}[
	xlabel={ \# Time Steps, $n$ },
	ylabel={Energy Law Error},
        legend pos=north east,
        legend style={cells={align=left}}]

        \addplot+[only marks,black,mark=*,mark options={fill=black}] coordinates{
			(1,9.281074e-02)
			(2,7.915581e-02)
			(4,5.530648e-02)
			(8,2.571345e-02)
			(16,5.417992e-03)
			(32,2.381807e-04)
			(64,4.566013e-07)
      };

       \addplot+[only marks,black,mark=square*,mark options={fill=black}] coordinates{
			(1,8.333160e-06)
			(2,8.241148e-06)
			(4,5.347711e-06)
			(8,2.401087e-06)
			(16,4.938247e-07)
			(32,2.093303e-08)
			(64,3.761648e-11)
      };

       \addplot+[only marks,black,mark=diamond*,mark options={scale=1.5,fill=black}] coordinates{
			(1,3.816432e-10)
			(2,3.376179e-10)
			(4,2.220855e-10)
			(8,1.000902e-10)
			(16,1.996719e-11)
			(32,9.557806e-13)
			(64,1.500380e-15)
      };
    \end{semilogyaxis}
  \end{tikzpicture}
  }
  \caption{}
  \label{fig:stokesvsn}
\end{subfigure}
\begin{subfigure}[b]{0.48\textwidth}
    \centering
    \resizebox{\linewidth}{!}{
    \begin{tikzpicture}
    \begin{semilogyaxis}[
	xlabel={ Time Step Size, $\tau$ },
	ylabel={Energy Law Error},
        legend pos=north east,
        legend style={cells={align=left}}]

        \addplot+[only marks,black,mark=*,mark options={fill=black}] coordinates{
		(0.001,1.026062e-01)
		(0.005,9.281074e-02)
		(0.01,8.325750e-02)
		(0.05,3.431861e-02 )
		(0.1,1.452381e-02 )
		(0.5, 5.006199e-04 )
		(1, 7.183924e-05  )
      };
      \addlegendentry{$p=1$}

       \addplot+[only marks,black,mark=square*,mark options={fill=black}] coordinates{
		(0.001,6.787202e-06)
		(0.005,8.333160e-06)
		(0.01,7.917273e-06 )
		(0.05,3.245971e-06 )
		(0.1,1.364182e-06 )
		(0.5, 4.642510e-08 )
		(1, 6.582738e-09 )
      };
      \addlegendentry{$p=2$}

       \addplot+[only marks,black,mark=diamond*,mark options={scale=1.5,fill=black}] coordinates{
		(0.001,3.330509e-10)
		(0.005,3.816432e-10)
		(0.01,3.420233e-10)
		(0.05,1.369607e-10)
		(0.1,5.730882e-11)
		(0.5,1.904810e-12)
		(1,1.164360e-12)
      };
      \addlegendentry{$p=3$}

    \end{semilogyaxis}
  \end{tikzpicture}
  }
  \caption{}
  \label{fig:stokesvstau}
  \end{subfigure}
\caption{Energy law error, \eqref{FOSLSenergylawstokes}, vs. (a) number of time steps, $n$ (with fixed
    $\tau=0.005$), and (b) time step
    size, $\tau$, for the
    FOSLS discretization of the
    Stokes' equation, \eqref{FOSLSstokes}, using various orders of the
    finite-element space ($p=1$ - linear; $p=2$ - quadratic; and $p=3$ -
    cubic). Mesh spacing is $h=\frac{1}{32}$.}
\label{fig:stokesvstime}
\end{figure}
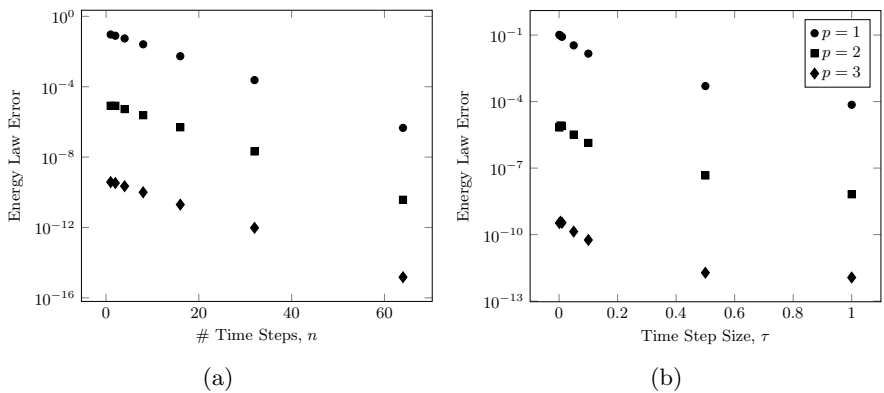

Figure \ref{fig:stokesvstime} again confirms how we expect the
timestepping to affect the convergence of the energy law.  Taking more
time steps decreases the error in the energy law, while the
convergence slightly worsens for small $\tau$.  These results also
highlight the similarities between the energy laws of the heat
equation and the time-dependent energy laws.  Since both have
underlying energy laws that are similar, the FOSLS discretization is
capable of capturing both with high accuracy.  

\section{Discussion: General Discrete Energy
  Laws}\label{sec:discussion}
The above results show that FOSLS discretizations of two specific PDEs
yield higher-order approximation of their underlying energy laws.  In
this section, we give a more general result, which suggests ideas for
extending this theory for other 
discrete energy laws using FOSLS discretizations.

\subsection{FOSLS Discrete Energy Laws}  
As encountered earlier, an energy law is an integral
relation of the form:
\begin{equation}\label{eq:testfunction}
  \left\langle \mathcal{L} u,u\right\rangle=0, \quad\mbox{for}\quad
  u(x,0)=u_0(x),
\end{equation}
where $\mathcal{L}: \widetilde{\mathcal{V}} \rightarrow \widetilde{\mathcal{V}}$ is a linear
operator (that involves boundary conditions),
$\widetilde{\mathcal{V}}$ is a function space, and
$u\in \widetilde{\mathcal{V}}$ is the solution to 
\begin{equation}\label{e:problem-space-time}
  \mathcal{L} u=0, \quad u(x,0)=u_0, \quad\mbox{for example:}\quad 
  \mathcal{L} = \partial_t  - \Delta. 
\end{equation}

To match the time-dependent problems considered in earlier sections,
  $\widetilde{\mathcal{V}}$ corresponds to a computational domain that
  involves both space and time, or as is often dubbed, a
  ``space-time'' domain: $\widetilde \Omega =
  \Omega\times[0,T]$. Further, we define a finite-dimensional space,
  $\widetilde V_h$ on $\widetilde{\Omega}$ corresponding to a
  triangulation of this space-time domain, as well as a ``stationary''
  finite-dimensional space, $V_h$, for $t=0$. Regarding such
  space-time discrete spaces and the related constructions, we refer
  the reader to the classical works by Johnson et
  al. \cite{2009JohnsonC-aa,1984JohnsonC_NavertU_PitkarantaJ-aa}, to
  \cite{1997MasudA_HughesT-aa} for space-time least squares
  formulations, and to \cite{2016LangerU_MooreS_NeumullerM-aa} for
  space-time iso-geometric analysis and a comprehensive literature
  review.

To present the FOSLS discretization in an abstract setting, we define
an extension of $u_0$ to the whole of $\widetilde{\Omega}$.  Without
loss of generality, we assume that the initial condition is a
piecewise polynomial and, more precisely, $u_0\in V_h$. Hence, we define
the extension $w_h\in \widetilde{\mathcal{V}}_h$ of $u_0$ so that
\(w_h(x,0)=u_0(x)\). This gives a non-homogenous problem with zero
initial guess, which is equivalent to \eqref{e:problem-space-time}. Its
weak form is: Find $u\in \widetilde{\mathcal{V}}$ such that for all
$v\in \widetilde{\mathcal{V}}_o$ there holds
\begin{equation}\label{e:weak-non-homogenous}
  u=\varphi + w_h, \quad\mbox{where}\quad    \left\langle\mathcal{L} \varphi , v\right\rangle = -\left\langle\mathcal{L} w_h,v\right\rangle, 
\end{equation}
Here, the space, $\widetilde{\mathcal{V}}_o$, is the subspace of
$\widetilde{\mathcal{V}}$ of functions with vanishing trace at $t=0$
(zero initial condition).  In a typical FOSLS setting, for the heat
equation, $u$ is a vector-valued function and the extension $w_h$
needs to be modified accordingly. We then have the following
space-time FOSLS discrete problem: Find
$u_h\in \widetilde{\mathcal{V}}_{h}$ such that for all
$v_h\in \widetilde{\mathcal{V}}_{h,o}$ there holds
\begin{equation}\label{dFOSLS}
  u_h = w_h+\varphi_h, \quad \mbox{where},\quad
  \left\langle\mathcal{L}\varphi_h,\mathcal{L}v_h\right\rangle =
  -\left\langle\mathcal{L} w_h,\mathcal{L} v_h\right\rangle.
\end{equation}

Restricting $\widetilde{\mathcal{V}}$ to a
finite-element space-time space,
$\widetilde{\mathcal{V}}_h \subset \widetilde{\mathcal{V}}$,
results in a restriction of $\mathcal{L}$ on
$\widetilde{\mathcal{V}}_h$, which is often called the ``discrete
operator''.

In the following, we keep $\left\langle\mathcal{L} u,u\right\rangle$ in all
estimates allowing for a non-homogenous right-hand side
in~\eqref{eq:testfunction}.  We now estimate the error in the energy
law, namely the difference
$\left\langle\mathcal{L} u,u\right\rangle - \left\langle\mathcal{L} u_h,u_h\right\rangle$.
\begin{theorem}\label{t:e-eh}
  If $u_h\in \widetilde{\mathcal{V}}_h$ is the 
  FOSLS solution of~\eqref{dFOSLS}. Then, the following estimate holds:
  \begin{equation}\label{e:e-eh}
    |\left\langle \mathcal{L} u,u\right\rangle-
    \left\langle \mathcal{L} u_h,u_h\right\rangle|\le C h^p \|u\|_{H^{p+1}}.
  \end{equation}
\end{theorem}
\begin{proof}
  For the left side of~\eqref{e:e-eh} we have
  \begin{eqnarray*}
    \left\langle\mathcal{L}u,u\right\rangle -\left\langle\mathcal{L}u_h,u_h\right\rangle
    & = &   \left\langle\mathcal{L}u,u\right\rangle - \left\langle\mathcal{L}u,u_h\right\rangle + \left\langle\mathcal{L}(u-u_h),u_h\right\rangle\\
 & = &
    \left\langle\mathcal{L}u,u-u_h\right\rangle +
    \left\langle\mathcal{L}(u-u_h),u_h\right\rangle.
  \end{eqnarray*}
  Using the continuity of $\mathcal{L}$ and the standard error estimates for
  the FOSLS discretization,
  \begin{eqnarray*}
    |\left\langle\mathcal{L}u,u\right\rangle - \left\langle\mathcal{L}u_h,u_h\right\rangle|
    & \le &            |\left\langle \mathcal{L}u,u-u_h\right\rangle| + 
            |\left\langle\mathcal{L}(u-u_h),u_h\right\rangle|\\
    & \le & C\|u-u_h\|_{H^1} \left(\|u\|_{H^1}+\|u_h\|_{H^1}\right)\\
    & \le & Ch^p \|u\|_{H^{p+1}}. 
  \end{eqnarray*}
  This concludes the proof. 
\end{proof}

\subsection{Exact Discrete Energy Law}  
Next, we provide a necessary and sufficient condition for the
FOSLS discretization to \emph{exactly} satisfy an energy law, namely conditions under which we have
\(\left\langle\mathcal{L} u_h,u_h\right\rangle =  \left\langle\mathcal{L} u,u\right\rangle\).
Recall the assumption that $u_0\in \mathcal{V}_h$.  Consider two
standard projections on the finite-element space,
$\widetilde{\mathcal{V}}_{h,o}$: (1) the Galerkin projection
$\Pi_h: \widetilde{\mathcal{V}}\mapsto \widetilde{\mathcal{V}}_{h,o}$;
and (2) the $L^2(\widetilde{\Omega})$-orthogonal projection,
$Q_h: L^2(\widetilde{\Omega}) \mapsto \widetilde{\mathcal{V}}_{h,o}$.
These operators are defined in a standard fashion:
 \begin{eqnarray*}
&&   \left\langle \mathcal{L} \Pi_h u,v_h \right\rangle := \left\langle \mathcal{L} u, 
   v_h \right\rangle, \quad \mbox{for all} ~ v_h \in \widetilde{\mathcal{V}}_{h,o} ~~ \mbox{and} ~~
   u \in \widetilde{\mathcal{V}},\\
&&   \left\langle Q_h u,v_h \right\rangle := \left\langle u, 
   v_h \right\rangle, \quad \mbox{for all} ~ v_h \in \widetilde{\mathcal{V}}_{h,o} ~~ \mbox{and} ~~
   u \in L^2(\widetilde{\Omega}).
 \end{eqnarray*}
Consider a well-known identity (see for example~\cite{1992XuJ-aa} for
 the case of symmetric $\mathcal{L}$) relating $\Pi_h$ and $Q_h$, which
 is used in~the later proof of Theorem~\ref{t:gelaw}.
\begin{lemma}\label{lem:project}
  The projections $Q_h$ and $\Pi_h$ satisfy the relation
  \begin{equation}\label{pq}
    \mathcal{L}_h\Pi_h = Q_h\mathcal{L},
  \end{equation}
  where $\mathcal{L}_h: \widetilde{\mathcal{V}}_h\mapsto\widetilde{\mathcal{V}}_h$ is the restriction of $\mathcal{L}$ on $\widetilde{\mathcal{V}}_h$, namely,
  \[
    \left\langle \mathcal{L}_hv_h,w_h\right\rangle =
    \left\langle \mathcal{L}v_h,w_h\right\rangle, \quad \mbox{for all}\quad
    v_h,\;w_h\in \widetilde{\mathcal{V}}_h.
  \]
\end{lemma}
\begin{proof}
  The result easily follows from the definitions of $Q_h$, $\Pi_h$,
  $\mathcal{L}_h$, and the fact that
  $\mathcal{L}_h\Pi_hv \in \widetilde{\mathcal{V}}_h$.  For
  $v\in \widetilde{\mathcal{V}}$, and $w\in \widetilde{\mathcal{V}}$
  we have
  \begin{eqnarray*}
    \left\langle \mathcal{L}_h\Pi_hv,w\right\rangle
    & = & \left\langle\mathcal{L}_h\Pi_hv,Q_hw\right\rangle =  \left\langle\mathcal{L}\Pi_hv,Q_hw\right\rangle\\
    & = & \left\langle\mathcal{L}v,Q_hw\right\rangle =  \left\langle Q_h\mathcal{L}v,Q_hw\right\rangle
                                               = \left\langle Q_h\mathcal{L}v,w\right\rangle.
  \end{eqnarray*}
This completes the proof. 
\end{proof}
Note that we use
$Q_h \chi_h = \Pi_h\chi_h = \chi_h$ for all
$\chi_h\in \widetilde{\mathcal{V}}_{h,o}$. In general, such an
identity is not true for $\chi_h\in
\widetilde{\mathcal{V}}_h$. However, we can relate the solution to~\eqref{dFOSLS} to a discrete analogue of
the energy law~\eqref{eq:testfunction} using Lemma~\ref{lem:project}.
Further, notice that the FOSLS solution, $u_h$,
satisfies $\left\langle\mathcal{L} u_h, \mathcal{L}\chi_h \right\rangle = 0$ only
for $\chi_h\in \widetilde{\mathcal{V}}_{h,o}$ corresponding to a zero
initial guess.  Thus, it is not obvious how to estimate
$\left\langle\mathcal{L} u_h,u_h\right\rangle-\left\langle\mathcal{L} u,u\right\rangle$. 
\begin{theorem}\label{t:gelaw}
  The solution $u_h$ of~\eqref{dFOSLS} satisfies the discrete energy
  law
  \(\left\langle \mathcal{L}u_h,u_h\right\rangle =\left\langle \mathcal{L}u,u\right\rangle\)
  if and only if there exists a $w_h\in \widetilde{\mathcal{V}}_h$
  satisfying the initial condition $w_h(x,0) = u_0(x)$ and if 
    \(\left\langle\mathcal{L} u_h,w_h \right\rangle = \left\langle\mathcal{L} u,u\right\rangle\).
\end{theorem}
\begin{proof}
 Let $w_h\in \widetilde{\mathcal{V}}_h$ be any extension of $u_0\in V_h$ in 
  $\widetilde{\Omega}$, that is, $w_h$ satisfies the initial condition. 
  The following relations follow directly from the definitions given
  earlier, Equation~\eqref{dFOSLS}, and Lemma~\ref{lem:project}.
  \begin{eqnarray*}
    \left\langle\mathcal{L} u_h, u_h \right\rangle & = & 
    \left\langle\mathcal{L} u_h, \underbrace{(u_h-w_h)}_{\in \widetilde{\mathcal{V}}_{h,o}}\right\rangle +  
    \left\langle\mathcal{L} u_h, w_h \right\rangle \\
& =&  
    \left\langle\mathcal{L} u_h, Q_h(u_h-w_h) \right\rangle + \left\langle\mathcal{L} u_h, w_h \right\rangle \\
& =&  
    \left\langle\mathcal{L} u_h, Q_h\mathcal{L}\mathcal{L}^{-1}(u_h-w_h) \right\rangle +      \left\langle\mathcal{L} u_h, w_h \right\rangle \\
& =  &
    \left\langle\mathcal{L} u_h, \mathcal{L}\underbrace{\Pi_h\mathcal{L}^{-1}(u_h-w_h)}_{v_h \in \widetilde{\mathcal{V}}_{h,o}}\right\rangle  +      \left\langle\mathcal{L} u_h, w_h \right\rangle = \left\langle\mathcal{L} u_h, w_h \right\rangle. 
  \end{eqnarray*}
  In the last identity, we use the fact that
  $v_h = \Pi_h\mathcal{L}^{-1}(u_h-w_h)$ is an element of $\widetilde{\mathcal{V}}_{h,o}$ and the first term on the right side vanishes (by Equation~\eqref{dFOSLS}). As a result, we have
  \begin{equation*}
    \left\langle\mathcal{L} u, u \right\rangle -
    \left\langle\mathcal{L} u_h, u_h \right\rangle 
    =
    \left\langle\mathcal{L} u, u \right\rangle -
    \left\langle\mathcal{L} u_h, w_h \right\rangle. 
  \end{equation*}
  which gives the desired necessary and sufficient condition.
\end{proof}
From the proof, we immediately obtain the following relation,
\begin{equation}\label{inf}
  |
  \left\langle\mathcal{L} u,u \right\rangle -
  \left\langle\mathcal{L} u_h,u_h \right\rangle 
  | = \inf_{w_h}\left\{ |    \left\langle\mathcal{L} u,u \right\rangle - \left\langle\mathcal{L} u_h,w_h \right\rangle|, \; 
    w_h(\cdot,0) = u_0 \right\}.
\end{equation}
In addition to the estimate in Theorem~\ref{t:e-eh}, it is plausible
that one can use the right side of~\eqref{inf} to obtain a sharper
result. While this is beyond the scope of this paper, some comments
are in order. The difficulties associated with each particular case in
hand (heat equation, Stokes' equation, etc.) amount to estimating the
quantity on the right side of~\eqref{inf} and such estimates depend on
the spaces chosen for discretization and how well the timestepping
approximates the space-time formulation.  Sharper estimates on the
error in discrete energy law, which uses~\eqref{inf}, can lead to
sharper bounds on the constant defined in~\eqref{e:truncation}.

\subsection{Conclusions}
In
this work, we have shown numerically that convergence of the discrete
energy law is of order
higher than the finite-element approximation order for two typical transient
problems. Thus, while it is known that the FOSLS method may have
issues with adherence to some conservation
laws (i.e., mass conservation), energy conservation is not such an issue, and can be
satisfied with high
accuracy. The rigorous theoretical justification of such claims are
topics of current and future research.

\section*{Acknowledgements}
The work of J.~H.~Adler was supported in part by NSF DMS-1216972.
I.~V.~Lashuk was supported in part by NSF DMS-1216972 (Tufts
University) and DMS-1418843 (Penn State).  S.~P.~MacLachlan was
partially supported by an NSERC Discovery Grant. The research of
L.~T.~Zikatanov was supported in part by NSF DMS-1720114 and the
Department of Mathematics at Tufts University.

\bibliographystyle{splncs}
\bibliography{mybib}

\end{document}